\documentclass[12pt,a4paper,fleqn,leqno]{amsart}
\usepackage{amssymb}
\usepackage{mathrsfs}
\theoremstyle{plain}
\newtheorem{theorem}{Theorem}[section]

\newtheorem{corollary}[theorem]{Corollary}
\newtheorem{proposition}[theorem]{Proposition}

\theoremstyle{definition}

\renewcommand{\emptyset}{\varnothing}

\DeclareMathOperator{\uhr}{\upharpoonright}

\DeclareMathOperator{\R}{\mathbb{R}}

\numberwithin{equation}{section} 

\begin{document}

\title{Sections, Selections and Prohorov's Theorem}

\author{Valentin Gutev}

\address{School of Mathematical Sciences, University of KwaZulu-Natal,
  Westville Campus, Private Bag X54001, Durban 4000, South Africa}

\email{gutev@ukzn.ac.za}


\thanks{Research of the first author is supported in part by the NRF
  of South Africa.}

\author{Vesko Valov}

\address{Department of Computer Science and Mathematics, Nipissing
  University, 100 College Drive, P.O. Box 5002, North Bay, ON, P1B
  8L7, Canada}

\email{veskov@nipissingu.ca}

\thanks{The second author was partially supported by NSERC Grant
  261914-03.}

\subjclass[2000]{Primary 54C60, 60B05; Secondary 54B20, 54C65, 28A33.}

\keywords{Set-valued mapping, lower semi-continuous, upper
  semi-continuous, selection, section, Radon probability measure.}

\begin{abstract}
  The famous Prohorov theorem for Radon probability measures is
  generalized in terms of usco mappings. In the case of completely
  metrizable spaces this is achieved by applying a classical Michael
  result on the existence of usco selections for l.s.c.\ mappings. A
  similar approach works when sieve-complete spaces are considered.
\end{abstract}

\date{\today}
\maketitle

\section{Introduction}
\label{section-introduction}

All spaces in this paper are assumed to be completely regular and
Hausdorff. For a space $X$, let $\mathscr{B}(X)$ be the Borel
$\sigma$-algebra associated to $X$, i.e.\ the smallest $\sigma$-algebra that
contains all closed subsets of $X$. Thus, $\mathscr{B}(X)$ is closed
with respect to complements and countable unions, its elements are
often called \emph{Borel} subsets of $X$.\medskip

A countably additive function $\mu:\mathscr{B}(X)\to [0,+\infty]$ is called a
\emph{Radon measure} on $X$ if
\begin{equation}
  \label{eq:radon-measure}
  \mu(B)=\sup\big\{\mu(K): K\subset B\ \text{and}\ K\ \text{is compact}\big\},\quad B\in
  \mathscr{B}(X).
\end{equation}
A \emph{Radon probability measure} is a Radon measure $\mu$, with
$\mu(X)=1$. In the sequel, we will denote by $\mathscr{P}(X)$ the set of
all Radon probability measures on $X$.  Every measure $\mu\in
\mathscr{P}(X)$ uniquely defines a positive linear functional $\mu(g)=\int
g d\mu$, where $g$ runs over the bounded continuous functions on $X$. As
a topological space, we consider $\mathscr{P}(X)$ endowed with the
weakest topology with respect to which all these functionals are
continuous. Thus, a net $\{\mu_\alpha\}\subset \mathscr{P}(X)$ converges to $\mu\in
\mathscr{P}(X)$ if and only if $\{\mu_\alpha(g)\}$ converges to $\mu(g)$ for
every bounded continuous function $g:X\to \R$. With respect to this
topology, for every closed $F\subset X$ and $\varepsilon>0$,
\begin{equation}
  \label{eq:open-set}
  \text{$\{\mu\in \mathscr{P}(X): \mu(F)<\varepsilon\}$ is open in
    $\mathscr{P}(X)$.}
\end{equation}

The famous Prohorov theorem \cite{prohorov:56} states that if $X$ is a
Polish space (i.e., a completely metrizable separable space), then for
every compact $T\subset \mathscr{P}(X)$ and every $\varepsilon > 0$ there exists a
compact $K\subset X$, with $\mu(X\setminus K) < \varepsilon$ for all $\mu\in T$. Spaces having this
property, called \emph{Prohorov spaces}, are widely investigated in
the literature.\medskip

In this paper, we give a simple proof that all sieve-complete spaces
are Prohorov (Theorem \ref{theorem-prohorov-sections}). In the special
case of completely metrizable spaces, this result follows by the
Michael theorem on the existence of usco selections for l.s.c.\
mappings, \cite[Theorem 1.1]{michael:59}. The general case of
arbitrary sieve-complete spaces follows by a selection-like result
\cite[Corollary 7.2]{gutev:07} which utilizes ``usco sections''
instead of ``usco selections''.\medskip

The idea to use some selection theorem for the proof of Prohorov's
theorem goes back to a question of Bouziad \cite{bouziad:01}. In fact,
our approach provides a natural generalization of Prohorov's theorem
in which the compact subset $T\subset \mathscr{P}(X)$ is replaced by a
paracompact one $Z\subset \mathscr{P}(X)$, and the compact $K\subset X$ --- by an
usco mapping from $Z$ into the compact subsets of $X$. This gives a
solution to another problem of Bouziad \cite{bouziad:01} whether there
is a ``continuous'' version of Prohorov's theorem, see Corollary
\ref{corollary-prohorov-selections}.  \medskip

The paper is organized as follows. Section \ref{section-lsc-mappings}
is devoted to the main ingredient of our approach which is a
construction of l.s.c.\ mappings generated by Radon probability
measures (Proposition \ref{proposition-generate-lsc}). Section
\ref{section-usco-selections} contains the proof of Theorem
\ref{theorem-prohorov-sections} which is preceded by that one for the
special case of completely metrizable spaces.

\section{A construction of l.s.c.\ mappings}
\label{section-lsc-mappings}

For a space $X$, let $2^X$ be the family of all nonempty subsets of
$X$, and let $\mathscr{C}(X)$ be the subfamily of $2^X$ which consists
of all compact members of $2^X$. A part of our considerations will
involve $\mathscr{C}(X)$ endowed with the \emph{Vietoris topology}
$\tau_V$. Recall that $\tau_V$ is generated by all collections of the form
$$
\langle\mathscr{V}\rangle = \left\{S\in \mathscr{C}(X) : S\subset \bigcup \mathscr{V}\ \
  \text{and}\ \ S\cap V\neq \emptyset,\ \hbox{whenever}\ V\in \mathscr{V}\right\},
$$
where $\mathscr{V}$ runs over the finite families of open subsets of
$X$. For convenience, for an open subset $V\subset X$, we write $\langle V\rangle$
rather than $\langle\{V\}\rangle$.\medskip

Another topology on $\mathscr{C}(X)$ that will play an important role
in this paper is the \emph{upper Vietoris topology} $\tau_V^+$, i.e.\ the
topology generated by the family
\[
\big\{\langle V\rangle: V\subset X\ \text{is open}\big\}.
\]
Clearly, $\tau_V^+$ is a coarser topology than the Vietoris one $\tau_V$,
i.e.\ $\tau_V^+\subset \tau_V$. In this regard, let us make the explicit agreement
that if $\tau$ is a topology on $\mathscr{C}(X)$, then the prefix
``$\tau$-'' will be used to express properties related to the topology
$\tau$, say $\tau$-open sets, $\tau$-closure, etc.\medskip

Finally, let us recall that a set-valued mapping $\Phi:Z\to 2^Y$ is
\emph{lower semi-continuous}, or l.s.c., if the set
\[
\Phi^{-1}(U)=\{z\in Z: \Phi(z)\cap U\neq \emptyset\}
\]
is open in $Z$ for every open $U\subset Y$.

\begin{proposition}
  \label{proposition-generate-lsc}
  Let $X$ be a space, and let $\varepsilon\in (0,1)$. Define a set-valued mapping
  $\Psi_\varepsilon:\mathscr{P}(X)\to 2^{\mathscr{C}(X)}$ by
  \[
  \Psi_\varepsilon(\mu)=\big\{K\in \mathscr{C}(X): \mu(X\setminus K)<\varepsilon\big\}, \ \mu\in \mathscr{P}(X).
  \]
  Then, $\Psi_\varepsilon$ is a nonempty-valued $\tau_V$-l.s.c.\ mapping.
\end{proposition}

\begin{proof}
  Take $\mu\in \mathscr{P}(X)$. Since $\mu(X)=1>1-\varepsilon$, by
  (\ref{eq:radon-measure}), there is $K\in \mathscr{C}(X)$ such that
  $\mu(K)>1-\varepsilon$, so $\Psi_\varepsilon(\mu)\neq\emptyset$. Let $K\in \Psi_\varepsilon(\mu)$ and let $\mathscr{V}$ be
  a finite family of open subsets of $X$, with $K\in \langle
  \mathscr{V}\rangle$. Then, $X\setminus \bigcup\mathscr{V}\subset X\setminus K$, it is closed in $X$
  and $\mu\left(X\setminus\bigcup \mathscr{V}\right)<\varepsilon$. Hence, by
  (\ref{eq:open-set}), there exists a neighbourhood $U$ of $\mu$ such
  that $\nu\left(X\setminus\bigcup \mathscr{V}\right)<\varepsilon$ for every $\nu\in U$. If $\nu\in U$,
  then $\nu\left(\bigcup \mathscr{V}\right)>1-\varepsilon$ and, by
  (\ref{eq:radon-measure}), there is a compact subset $H\subset \bigcup
  \mathscr{V}$, with $\nu(H)>1-\varepsilon$. We now have that $H\cup K\in \langle
  \mathscr{V}\rangle$, while $H\cup K\in\Psi_\varepsilon(\nu)$ because $\nu\big(X\setminus (H\cup K)\big)\leq
  \nu(X\setminus H)<\varepsilon$.
\end{proof}

\begin{proposition}
  \label{proposition-closure-control}
  Let $X$ be a space, $\varepsilon\in (0,1)$, $\Psi_\varepsilon:\mathscr{P}(X)\to
  2^{\mathscr{C}(X)}$ be defined as in Proposition
  \ref{proposition-generate-lsc}, and let $\Phi_\varepsilon(\mu)$ be the
  $\tau_V^+$-closure of $\Psi_\varepsilon(\mu)$, for each $\mu\in \mathscr{P}(X)$. Then,
  $\mu(X\setminus K)\leq \varepsilon$ for every $K\in \Phi_\varepsilon(\mu)$ and $\mu\in \mathscr{P}(X)$.
\end{proposition}

\begin{proof}
  Take $\mu\in \mathscr{P}(X)$ and $K\in \mathscr{C}(X)$ such that $\mu(X\setminus
  K)>\varepsilon$. By (\ref{eq:radon-measure}), there exists a compact subset
  $H\subset X\setminus K$, with $\mu(H)>\varepsilon$. Let $V=X\setminus H$. We now have that $K\in \langle V\rangle$,
  while $\varepsilon<\mu(H)=\mu(X\setminus V)\leq \mu(X\setminus S)$ for every $S\in \langle V\rangle$. Consequently,
  $K\notin \Phi_\varepsilon(\mu)$ because $\Psi_\varepsilon(\mu)\subset \mathscr{C}(X)\setminus \langle{V}\rangle$.
\end{proof}

We conclude this section with a well-known property of compact sets in
the upper Vietoris topology.

\begin{proposition}
  \label{proposition-compact-upper}
  Let $\mathscr{K}\subset \mathscr{C}(X)$ be a $\tau_V^+$-compact set. Then, $\bigcup
  \mathscr{K}$ is compact in $X$.
\end{proposition}

\begin{proof}
  Take an open in $X$ cover $\mathscr{U}$ of $\bigcup\mathscr{K}$. Then,
  $\Omega=\big\{\langle \bigcup\mathscr{E}\rangle: \mathscr{E}\subset \mathscr{U}\ \text{is
    finite}\big\}$ is a $\tau_V^+$-open cover of $\mathscr{K}$. Hence, $\Omega$
  contains a finite subcover of $\mathscr{K}$, so there exists a
  finite $\mathscr{V}\subset \mathscr{U}$, with $\mathscr{K}\subset \bigcup \big\{\langle
  \bigcup\mathscr{E}\rangle: \mathscr{E}\subset \mathscr{V}\ \text{is
    finite}\big\}$. This $\mathscr{V}$ is a finite cover of
  $\bigcup\mathscr{K}$.
\end{proof}

\section{Usco mappings and Prohorov's theorem}
\label{section-usco-selections}

Recall that a set-valued mapping $\psi :Z\to 2^X$ is \emph{upper
  semi-continuous}, or u.s.c., if the set
\[
\psi^{\#}(U)=\{z\in Z: \psi(z)\subset U\}
\]
is open in $Z$ for every open $U\subset X$. We say that $\psi:Z\to 2^X$ is
\emph{usco} if it is u.s.c.\ and compact-valued. Let us explicitly
mention that if $\psi:Z\to \mathscr{C}(X)$ is usco, then $\psi(T)=\bigcup\{\psi(z):z\in
T\}$ is compact for every compact $T\subset Z$. \medskip

A space $X$ is \emph{sieve-complete} \cite{chaber-choban-nagami:74} if
it has an open complete sieve. Every {\v C}ech-complete space is
sieve-complete, and it was shown in \cite{chaber-choban-nagami:74}
(see, also, \cite{michael:77}) that the two concepts are equivalent in
the presence of paracompactness.

\begin{theorem}
  \label{theorem-prohorov-sections}
  Let $X$ be a sieve-complete space, and let $Z\subset \mathscr{P}(X)$ be
  paracompact. Then, for every $\varepsilon>0$ there is an usco mapping $\varphi:Z\to
  \mathscr{C}(X)$ such that $\mu(X\setminus \varphi(\mu))<\varepsilon$ for every $\mu\in Z$.
\end{theorem}

Turning to the proof of Theorem \ref{theorem-prohorov-sections}, let
us first demonstrate the special case of a completely metrizable
$X$. In this case, let $\Psi_\varepsilon :\mathscr{P}(X)\to 2^{\mathscr{C}(X)}$ be
defined as in Proposition \ref{proposition-generate-lsc}, and let
$\Phi(\mu)$ be the $\tau_V$-closure of $\Psi_\varepsilon(\mu)$, for each $\mu\in
\mathscr{P}(X)$. By Proposition \ref{proposition-generate-lsc} and
\cite[Proposition 2.3]{michael:56a}, $\Phi:\mathscr{P}(X)\to
2^{\mathscr{C}(X)}$ is $\tau_V$-l.s.c.  Also, $(\mathscr{C}(X),\tau_V)$ is
completely metrizable because so is $X$,
\cite{hausdorff:14,kuratowski:56,michael:51}. Hence, by \cite[Theorem
1.1]{michael:59}, $\Phi\uhr Z$ has a $\tau_V$-usco selection $\theta:Z\to
2^{\mathscr{C}(X)}$. That is, $\theta$ is a $\tau_V$-usco mapping such that
$\theta(\mu)\subset \Phi(\mu)$ for every $\mu\in Z$. Then, define $\varphi:Z\to \mathscr{C}(X)$ by
letting $\varphi(\mu)=\bigcup \theta(\mu)$, $\mu\in Z$. This $\varphi$ is as required. Indeed, each
$\theta(\mu)$, $\mu\in Z$, is $\tau_V$-compact, hence $\tau_V^+$-compact as well, and,
by Proposition \ref{proposition-compact-upper}, each $\varphi(\mu)$, $\mu\in Z$,
is a compact subset of $X$. If $V$ is a neighbourhood of $\varphi(\mu)$ for
some $\mu\in Z$, then $\langle V\rangle$ is a neighbourhood of $\theta(\mu)$. This implies
that $\varphi$ is u.s.c. Finally, take $\mu\in Z$ and $K\in \theta(\mu)\subset \Phi(\mu)$. Since
$\tau_V^+\subset \tau_V$, we have that $\Phi(\mu)$ is a subset of the $\tau_V^+$-closure
of $\Psi_\varepsilon(\mu)$. Therefore, by Proposition
\ref{proposition-closure-control}, $\mu\big(X\setminus \varphi(\mu)\big)\leq \mu(X\setminus K)\leq \varepsilon$
because $K\subset \varphi(\mu)$.\medskip

The proof of Theorem \ref{theorem-prohorov-sections} for the general
case of arbitrary sieve-complete spaces follows exactly the same idea
but is now based on the upper Vietoris topology and another
selection-like result for usco mappings.

\begin{proof}[Proof of Theorem \ref{theorem-prohorov-sections}]
  Let $X$ and $Z\subset \mathscr{P}(X)$ be as in that theorem, and let $\varepsilon\in
  (0,1)$. Also, for each $\mu\in \mathscr{P}(X)$, let $\Phi_\varepsilon(\mu)$ be the
  $\tau_V^+$-closure of $\Psi_\varepsilon(\mu)$, where $\Psi_\varepsilon:\mathscr{P}(X)\to
  2^{\mathscr{C}(X)}$ is defined as in Proposition
  \ref{proposition-generate-lsc}. By Proposition
  \ref{proposition-generate-lsc} and \cite[Proposition
  2.3]{michael:56a}, $\Phi_\varepsilon:\mathscr{P}(X)\to 2^{\mathscr{C}(X)}$ is
  $\tau_V^+$-l.s.c.\ because $\tau_V^+\subset \tau_V$. By \cite[Lemma
  3.1]{nedev-pelant-valov-09}, $(\mathscr{C}(X),\tau_V^+)$ is
  sieve-complete because so is $X$. Hence, by \cite[Corollary
  7.2]{gutev:07}, $\Phi_\varepsilon\uhr Z$ has a $\tau_V^+$-usco section $\theta:Z\to
  2^{\mathscr{C}(X)}$. That is, $\theta$ is a $\tau_V^+$-usco mapping such
  that $\theta(\mu)\cap \Phi_\varepsilon(\mu)\neq\emptyset$ for every $\mu\in Z$. Finally, define the required
  $\varphi:Z\to \mathscr{C}(X)$ by $\varphi(\mu)=\bigcup \theta(\mu)$, $\mu\in Z$. By Proposition
  \ref{proposition-compact-upper}, each $\varphi(\mu)$, $\mu\in Z$, is a compact
  subset of $X$. Just like before $\varphi$ is u.s.c.\ because if $V$ is a
  neighbourhood of $\varphi(\mu)$ for some $\mu\in Z$, then $\langle V\rangle$ is a
  neighbourhood of $\theta(\mu)$. Finally, if $\mu\in Z$ and $K\in \theta(\mu)\cap \Phi_\varepsilon(\mu)$,
  then, by Proposition \ref{proposition-closure-control}, $\mu\big(X\setminus
  \varphi(\mu)\big)\leq \mu(X\setminus K)\leq \varepsilon$ because $K\subset \varphi(\mu)$. The proof is completed.
\end{proof}

It is well-known that $\mathscr{P}(X)$ is paracompact (and
{\v C}ech-complete) whenever $X$ is so,
\cite{banakh:95,topsoe:72,topsoe:74}, see also \cite{choban:06}. This
gives the following immediate consequence.

\begin{corollary}
  \label{corollary-prohorov-selections}
  Let $X$ be a paracompact {\v C}ech-complete space, and $\varepsilon>0$. Then,
  there is an usco mapping $\varphi:\mathscr{P}(X)\to \mathscr{C}(X)$ such
  that $\mu(X\setminus \varphi(\mu))<\varepsilon$ for every $\mu\in \mathscr{P}(X)$. In particular,
  $\Phi(T)=\bigcup\{\varphi(\mu):\mu\in T\}$, $T\in \mathscr{C}(\mathscr{P}(X))$, defines a
  continuous map $\Phi:\big(\mathscr{C}(\mathscr{P}(X)),\tau_V^+\big)\to
  \big(\mathscr{C}(X),\tau_V^+\big)$ such that $\mu(X\setminus \Phi(T))<\varepsilon$ for every
  $T\in \mathscr{C}(\mathscr{P}(X))$ and $\mu\in T$.
\end{corollary}

\newcommand{\noopsort}[1]{} \newcommand{\singleletter}[1]{#1}
\providecommand{\bysame}{\leavevmode\hbox to3em{\hrulefill}\thinspace}
\providecommand{\MR}{\relax\ifhmode\unskip\space\fi MR }
\providecommand{\MRhref}[2]{%
  \href{http://www.ams.org/mathscinet-getitem?mr=#1}{#2}
}
\providecommand{\href}[2]{#2}


\end{document}